\documentclass[a4paper]{amsart}
\usepackage[T1]{fontenc}
\usepackage{enumerate}
\usepackage{amsthm}
\usepackage{amscd}
\usepackage{graphicx}
\usepackage{lmodern}
 \usepackage{pdfsync}
\def\car{{\mathbf 1}}

\def\/{\, | \,}
\def\e{\varepsilon}
\def\ee{\e}
\def\C{C}
\def\D{{\mathbb D}}
\def\R{{\mathbf R}}
\def\L{{\mathbf L}}
\def\P{{\mathbf P}}
\def\F{{\mathcal F}}
\def\Proba0{{\mathcal P}_0}
\newcommand{\la}{\langle\,}
\def\ra{\,\rangle}
\def\<<{\langle\!\langle\,}
\def\>>{\,\rangle\!\rangle}
\def\d{\, \text{d}}
\newcommand{\esp}[1]{{\mathbf E}\left[{#1}\right]}

\newcommand{\N}{{\mathbf N}}

\newtheorem{theorem}{Theorem}[section]

\newtheorem{lemma}{Lemma}[section]
\newtheorem{defn}{Definition}

\begin{document}
\title{A {M}arkov model for the spread of Hepatitis C virus}
\author{L. Coutin}
\address{Institut de Mathématiques de Toulouse, LSP\\
  Toulouse, France} \email{coutin@cict.fr}
\author{L. Decreusefond}
\address{Institut Telecom, Telecom ParisTech, CNRS LTCI\\
  Paris, France} \email{Laurent.Decreusefond@telecom-paristech.fr}
\thanks{This work was carried out during a stay of Laurent
  Decreusefond at Universit\'e Paris Descartes. He would like to thank
  Universit\'e Paris Descartes for warm hospitality.}
\author{J.S. Dhersin}
\address{Department of Mathematics\\
  Institut Galil\'ee\\
  Paris, France}
\email{dhersin@math.univ-paris13.fr}
\begin{abstract}
  We propose a Markov model for the spread of Hepatitis C virus (HCV)
  among drug users who use injections. We then proceed to an
  asymptotic analysis (large initial population) and show that the
  Markov process is close to the solution of a non linear autonomous
  differential system. We prove both a law of large numbers and
  functional central limit theorem to precise the speed of convergence
  towards the limiting system. The deterministic system itself
  converges, as time goes to infinity, to an equilibrium point. This
  corroborates the empirical observations about the prevalence of HCV.
\end{abstract}
\maketitle{}
\section{Motivations}
\label{sec:motivations}
Hepatitis C virus  (HCV) infects 170 million people in the world (3~\%
of the population) and 9 million in Europe (1~\% of the population)
\cite{who}. More than 75~\% of newly infected patients progress to
develop chronic infection. Then, Cirrhosis develops in about 10~\% to
20~\%, and liver cancer develops in 1~\% to 5~\% over a period of 20
to 30 years. These long-term consequences, which suggest an increased
mortality due to HCV infection, make the prevention of spread of
hepatitis C a major public health concern.

HCV is spread primarily by direct contact with human blood. In
developed countries that have safe blood supplies, the population
infected by HCV is closely related to injecting drug users (IDU).  It
is estimated that 90~\% of infectious are due to IDU
\cite{2004_HPA}. In order to reduce the numbers of new hepatitis C cases, preventing
infections in IDU is then a priority. Programs exist all over the
world which try to reduce the prevalence of many infectious diseases
like HIV or hepatitis C, among injecting drug users. They are mainly
 based on needle exchanges. It turns out that after several years of such
 programs, the HIV prevalence seems to be now rather low whereas the
 percentage of IDU who are HCV positive remains about 60~\% \cite{2006_Jauffret-Roustide_Substance_Use_and_Misuse,2004_HPA}. We
 were asked by epidemiologists to provide them a mathematical model
 which could quantitatively evaluate  the differences between the two
 diseases.

It is always a challenge to analyze an epidemic problem because there
are so many real-life situations that should be incorporated while keeping the
mathematical  model tractable. Moreover, epidemic field  studies are expensive and  hard to organize so that parameter estimates are rare and often
imprecise. It is thus necessary to deal with parsimonious models whose
parameters have clear and visible meaning.
To the best of our knowledge, the only models which have been developed
for the dynamics of HCV transmission are found in the references \cite{2007_Vickerman_IJE,esposito2004nem}. It is a deterministic model with more than
twenty-five parameters, for which the authors do not have explicit results for the asymptotics and
only estimate them by simulations.  In our paper, we propose a
parsimonious Markovian model for the spread of HCV in a local population of
IDU. It should be noted that our model bears some resemblance to a random SIR
(Susceptible-Infected-Recovered) model but differs from  it by some essential
characteristics. Our Susceptible (respectively Infected) are IDU who
are sero-negative (respectively sero-positive). There is no Recovered
category in our model since we can't measure their number (when they
are no longer IDU, they can't be counted in studies focused on drugs
users). Moreover, our population is not closed (there are new
susceptible all the time) and a new drug user may be infected at his
first injection. This means that there is an exogeneous flow to the
Infected category, a feature which is not included in usual SIR models.

To keep the Markovian character of our model, we made the following
usual and reasonable hypothesis.  Exogenous antibody-positive and antibody-negative individuals
arrive in this local population according to Poisson processes.  If
initiated by an antibody-positive drug addict, a new IDU acquires the virus
very rapidly after the initiation
\cite{1990_Bell_MedJAust,1996_Garfein_AmJPublicHealth}. HCV then
spreads in the population by sharing syringe, needles and other
accessories (cotton, boilers, etc.). Each individual of the population
stays in his state (infected/non infected)  for an exponentially distributed
time. We present the model in Section 3. If we denote by $X_1(t)$
(resp. $X_2(t)$) the number of antibody-positive
(resp. antibody-negative) individuals in the local population at time
$t$, we prove that the process $X=(X_1,X_2)$ is an ergodic Markov
process. In Section 4., we give a related deterministic differential
system connected with this Markov process. We study its asymptotic
behaviour and give an explicit expression of the limit of the
solution. In Section 5., we give a mean-field approximation of the
process $X$: For large populations, we prove that the process $X$ is
close to the solution ${\psi}$ of the deterministic differential
system. In Section 6., we prove that, for large populations, the
invariant distribution for the Markov process $X$ can be approximated
by the Dirac measure which only charges ${\psi}(\infty )$. Hence we
can give an explicit limit of the prevalence of HCV in the
population. In Section \ref{sec:centr-limit-theor}, we give a central limit theorem
for the approximation of $X$ by ${\psi}$ when the population tends to
infinity. In Section \ref{sec:numer-invest}, we show that even for a small value of $N$, there is a good accordance between the prevalence
computed on the deterministic limit and the prevalence observed on the
stochastic model. We also show that this can be extended to the
sensitivity of the model with respect to slight variations of some parameters.
\section{Preliminaries}
\label{sec_preliminaries}
Let us denote by $\D([0,T],\R^2)$ the set of cadlag functions 
equipped with its usual topology. In this Section, we recall some
results about cadlag semi-martingales; for details we refer to
\cite{MR982268}. We assume that we are given $(\Omega,\, (\F_t,\, t\ge
0),\, \P)$ a filtered probability space satisfying the so-called usual
hypothesis. On $(\Omega,\, (\F_t,\, t\ge 0),\, \P)$, let $X$ and $Y$
be two real-valued cadlag square integrable semi-martingales. The
mutual variation of $X$ and $Y$, denoted by $[X,\, Y]$, is the right
continuous process with finite variation such that the following
integration by parts formula is satisfied:
\begin{equation*}
  X(t)Y(t)-X(0)Y(0)=\int_{(0,\, t]} X(s_-)\d Y(s)+\int_{(0,\, t]} Y(s_-)\d X(s)+[X,\, Y]_t.
\end{equation*}
The Meyer process of the couple $(X,\, Y)$, or its square bracket, is
denoted by $\la X,\, Y \ra$  is the unique right continuous with
finite variation predictable process such that
\begin{equation*}
  X(t)Y(t)-X(0)Y(0)-\la X,\, Y\ra_t
\end{equation*}
is a martingale. Alternatively, $\la X,\, Y\ra$ and is the unique
right continuous, predictable with finite variation, process such that
$[X,\, Y]-\la X,\, Y\ra$ is a martingale.  For a vector valued
semi-martingale $X=(X_1,\, X_2)$ where $X_1$ and $X_2$ are real valued
martingales, we denote by $\<< X \>>$, its square bracket, defined by
\begin{equation*}
  \<< X \>>_t=
  \begin{pmatrix}
    \la X_1\ra_t & \la X_1,\, X_2\ra_t\\
    \la X_1,\, X_2\ra_t & \la X_2\ra_t
  \end{pmatrix}.
\end{equation*}
In the sequel, if $x$ is a vector (resp. $M$ a matrix) we denote by
$\|x\|$ (resp. $\|M\|$) its $\L^1$-norm.

Let $E$ be a discrete denumerable space. Let $(X(t),\, t\ge 0)$ be an
$E$-valued, pure jump Markov process, with infinitesimal generator
$Q=(q_{xy},\, (x,\, y)\in E\times E)$. For any
$F\,:\,E\to\R$,  Dynkin's Lemma states that the process:
\begin{equation*}
  F(X(t))-F(X(0))-\int_0^t QF(X_s)\d s
\end{equation*}
is a local martingale, where
\begin{equation*}
  QF(x)=\sum_{y\neq x} (F(y)-F(x))q_{xy}.
\end{equation*}
Here and hereafter, we identify the matrix $Q$ and the operator $Q$
defined as above.
\section{Markov model}
\label{sec:markov-model}
We consider the dynamics of HCV among a local population which suffers
a continuous arrival of exogenous antibody-positive individuals,
described by a Poisson process of intensity $r$. We let $X_1(t)$ and
$X_2(t)$ denote the number of antibody-positive, respectively
antibody-negative, users at time $t$ in the population under
consideration. The new  susceptible drug users arrive as a Poisson process
of intensity $\lambda$. We assume that for their first injection, they
are initiated by an older IDU who has a probability
$q(t)=X_1(t)(X_1(t)+X_2(t))^{-1}$ of being infected. For different
reasons, even in this situation, the probability of being infected,
 is not exactly one and is denoted by $p_I$. Each
time, an antibody-negative IDU has an injection, he may share some of
his paraphernalia and may become infected if the sharing occurs with
an infected IDU. We summarize all these probabilities by saying that
at each injection, the probability of becoming infected is
\begin{math}
  p q(t),
\end{math}
where $p$ is a parameter to be estimated, as is $p_I$. If we denote by
${\alpha}$ the rate at which an IDU injects, and if ${\alpha}p$ is
small, we can assume that the rate at which a sane IDU in the
population is infected, is given by ${\alpha}pq(t)$. Once infected, an IDU may
exit from the population under consideration either by a death, self
healing or stopping drug usage. The whole of these situations is
modeled by an exponentially distributed duration with parameter
$\mu_1$. For antibody-negative IDU, the only way to exit the
population is by stopping drug injection, supposed to happen after an
exponentially distributed duration with parameter $\mu_2$. In summary,
the transitions are described in Figure~\ref{fig:transitions}.

\begin{figure}[!ht]
 \begin{center}
   \leavevmode
    \includegraphics*[width=0.8\textwidth]{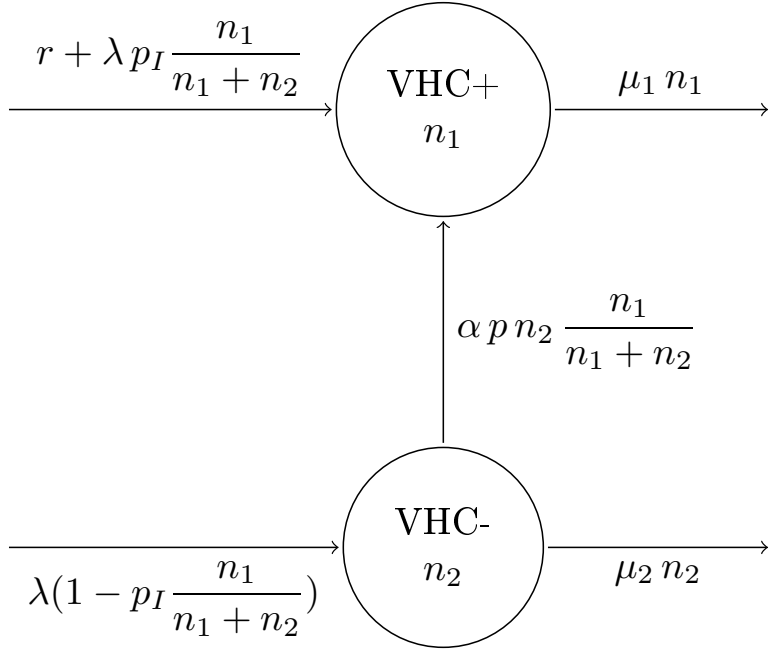}
   \caption{Transitions of the Markov model.}
    \label{fig:transitions}
  \end{center}
\end{figure}

For further references, we set
\begin{align*}
  q_1(n_1,\, n_2)&=r + \lambda\,p_I\dfrac{n_1}{n_1+n_2}\\
  q_2(n_1,\, n_2)&=\mu_1\, n_1\\
  q_3(n_1,\, n_2)&=\alpha\, p \,n_2\,\dfrac{n_1}{n_1+n_2}\\
  q_4(n_1,\, n_2)&=\lambda(1-p_I\dfrac{n_1}{n_1+n_2})\\
  q_5(n_1,\, n_2)&=\mu_2\, n_2 .
\end{align*}
\begin{lemma}
  \label{lem:borne_sup} Let $x^0=(x_1^0,x_2^0)$. Conditionally on
  $X(0)=x^0$, the process $W(t)=X_1(t)+X_2(t)-(x_1+x_2)$ is dominated
  (for the strong stochastic order of processes) by a Poisson process
  of intensity $r+\lambda$.  In particular, for any $t\in [0,T]$,
  \begin{equation*}
    \esp{\sup_{t\le T}\|X(t)\|^p\, \Big| \, X(0)=x^0}\le (\|x^0\|+ (r+\lambda)T)^p,
  \end{equation*}
for any $p\ge 1$.
\end{lemma}
\begin{proof}
  It suffices to say that by suppressing all the departures, we get
  another system with a population larger than that of the system
  under consideration, at any time, for any trajectory. Then,
  $X^N_1(t)+X^N_2(t)-(x_1+x_2)$ is less than the number of arrivals of
  a Poisson process of intensity $r+\lambda$. Since a Poisson process
  has increasing path,  its supremum over $[0,\, T]$
  is  its value at time $T$. The second assertion
  follows.
\end{proof}
\begin{theorem}\label{thm:ergodic}
  The Markov process $X=(X_1,\, X_2)$ is ergodic. For $r>0$, the
  process $X$ is irreducible. For $r=0$, the set $\{(n_1,\, n_2)\in
  \N\times \N,\ n_1=0\}$ is a proper closed subset.
\end{theorem}
\begin{proof}
  Let $S$ be the function defined on $\N\times \N$ by
  \begin{equation*}
    S(n_1,\ n_2)=\|(n_1,n_2)\|=n_1+n_2.
  \end{equation*}
  If we denote by $Q$ the infinitesimal generator of $X$, we have
  \begin{equation*}
    QS(n_1,\, n_2)=\lambda+r-\mu_1 n_1-\mu_2 n_2.
  \end{equation*}
  Let $K$ be a real strictly greater than $(\lambda+r+1)/\mu_-$ where
  $\mu_-=\mu_1\wedge \mu_2$ and consider the following finite subset
  of the state space:
  \begin{equation*}
    D_K=\{(n_1,\, n_2)\in \N\times \N,\ n_1+n_2\le K\}.
  \end{equation*}
  If $(n_1,\, n_2)$ belongs to $D_K^c$, then
  \begin{equation*}
    QS(n_1,\, n_2)\le \lambda+r-\mu_-( n_1+ n_2)<-1.
  \end{equation*}
  Lemma \ref{lem:borne_sup} implies that both
  \begin{equation*}
    \esp{\sup_{s\in[0,\, 1]} S(X(s))} \text{ and } \esp{\int_0^1 |QS(X(s))|\d s}
  \end{equation*}
  are finite.  Then according to \cite[Proposition 8.14]{MR1996883},
  $X$ is ergodic.

  The second and third assertions are immediate through inspection of
  the transition rates.
\end{proof}
With the non-linearity appearing in the transitions, it seems hopeless
to find an exact expression for the stationary probability of the
Markov process $(X_1,\, X_2)$. As usual in queueing theory
\cite{MR1996883}, we then resort to asymptotic analysis in order to
gain some insights on the evolution of this system. This means that we
let the initial population becoming larger and larger. For keeping
other quantities of the same order of magnitude, one are thus led to
increase $r$ and $\lambda$ at the same speed, i.e., keeping the ratio
$i=(r+\lambda)/(x_1+x_2)$ constant. Note that in epidemiological
language, $i$ is the incidence of new susceptible. It is measured in
percentage of individuals per unit of year.
\section{A deterministic differential system}
\label{sec:differential-system}
The mean field approximation will lead us to investigate the solutions of the
following differential system with initial condition
$x^0=(x^0_1,x^0_2)\in (\R_+\times \R_+)\backslash
\left\{(0,0)\right\}$:
\begin{equation}
  \label{eq:1}
  \tag{$S_r(x^0)$}
  \begin{cases}
    \psi_1^\prime(t)&= r + \lambda\,p_I\dfrac{\psi_1(t)}{\psi_1(t)+\psi_2(t)}-\mu_1\, \psi_1(t)+\alpha\, p     \,\dfrac{\psi_1(t)\psi_2(t)}{\psi_1(t)+\psi_2(t)},\\
    \psi_1(0)&=x_1^0,\\
    \psi_2^\prime(t)&=\lambda(1-p_I\dfrac{\psi_1(t)}{\psi_1(t)+\psi_2(t)})-\mu_2\, \psi_2(t)-\alpha\,     p\,\dfrac{\psi_1(t)\psi_2(t)}{\psi_1(t)+\psi_2(t)},\\
    \psi_2(0)&=x_2^0.
  \end{cases}
\end{equation}
\begin{theorem}
  \label{thm:asymptotique_syst_diff}
  For any $x^0=(x^0_1,x^0_2)\in (\R_+\times \R_+)\backslash
  \left\{(0,0)\right\}$, there exists a unique solution to
  (\ref{eq:1}). Furthermore, this solution is defined on $\R$. For
  $r>0$, the differential system has a unique fixed point $(\xi_1,\,
  \xi_2)$ in $\R_+\times \R_+,$ defined by the equations
  \begin{equation}
    \label{eq:2}
\xi_2=\frac{1}{\mu_2}(r+\lambda -\mu_1\xi_1)\text{ and } \xi_1=\frac{ab-c+\text{sgn}(a)\sqrt{(ab-c)^2+4abr\mu_1}}{2a\mu_1},
  \end{equation}
where $a=\alpha p -\mu_1+\mu_2,\ b=r+\lambda\text{ and } c=r\mu_1+\lambda(1-p_I)\mu_2.$
Moreover, for $r>0$ and any $x^0\in \R_+^2\backslash \{0,\, 0\}$,
\begin{equation*}
   \lim_{t\to+\infty}(\psi_1(t),\psi_2(t))=(\xi_1,\, \xi_2).
\end{equation*}
  If $r=0$ and $x^0_1=0$ then
  \begin{equation*}
    \psi_1(t)=0 \text{ for all } t \text{ and }  \lim_{t\to+\infty}(\psi_1(t),\psi_2(t))=(0,\lambda/ \mu_2).
  \end{equation*}
  If $r=0$ and $\rho=\alpha p +\mu_2p_I-\mu_1>0$, then there exists
  two equilibrium points: one is $(0,\,\lambda/\mu_2)$ and the other
  is the unique solution with positive first coordinate of
  (\ref{eq:2}).  If $x^0_1>0$ then
  \begin{equation*}
    \lim_{t\to+\infty}(\psi_1(t),\psi_2(t))=(\xi_1,\, \xi_2).
  \end{equation*}
  If $r=0$ and $\rho\le 0$, then for any $x^0$ with positive $x^0_1$,
  \begin{equation*}
    \lim_{t\to+\infty}(\psi_1(t),\psi_2(t))=(0,\lambda/ \mu_2).
  \end{equation*}
  For further references, we denote by $\psi^\infty$ the unique point
  to which the system converges in each case.  We denote by $\Psi$ the
  measurable function such that
  \begin{math}
    \Psi(x^0,\, t)
  \end{math}
  is the value of the solution of (\ref{eq:1}) at time $t$.
\end{theorem}
\begin{proof}
  We denote by $f_1$ and $f_2$ the functions such that (\ref{eq:1}) is
  written
  \begin{equation}\label{eq:12}
    \psi_1^\prime(t)=f_1(\psi_1(t),\, \psi_2(t)) \text{ and }  \psi_1^\prime(t)=f_2(\psi_1(t),\, \psi_2(t)).
  \end{equation}
  Since $f_1$ and $f_2$ are locally Lipschitz, there exists a local
  solution for any starting point $x^0$ belonging to $(\R_+\times
  \R_+)\backslash \left\{(0,0)\right\}$. Moreover, for any $(x_1,\,
  x_2)\in (\R_+\times \R_+)\backslash \left\{(0,0)\right\}$,
  \begin{equation*}
    r -\mu_1 x_1\le f_1(x_1,x_2)\le r+\lambda p_I+\alpha p x_1 \text{ and }
    \lambda(1-p_I)-\mu_2x_2\le f_2(x_1,x_2)\le \lambda.
  \end{equation*}
  By standard theorems about comparison of solutions of differential
  equations, one can then show that every local solution ${\psi}$ can be
  extended to $\R$ and that for any $t\in \R$, $\psi(t)=(\psi_1(t),\,
  \psi_2(t))$ belongs to $(\R_+\times \R_+)\backslash
  \left\{(0,0)\right\}$.  Furthermore, with direct calculations, we
  have
  \begin{equation}\label{eq:3}
    \frac{\d}{\d t}(\psi_1(t)+\psi_2(t))=r+\lambda-\mu_1\psi_1(t)-\mu_2\psi_2(t).
  \end{equation}
  For $\ee>0$, consider

  \begin{align*}
    A_\pm^\ee&=\{(x_1,\, x_2)\in \R_+\times \R_+,\ 0\leq \pm( r+\lambda-\mu_1x_1-\mu_2x_2)< \ee\},\\
    B_+^\ee&=\{(x_1,\, x_2)\in \R_+\times \R_+,\  r+\lambda-\mu_1x_1-\mu_2x_2\geq  \ee\},\\
    B_-^\ee&=\{(x_1,\, x_2)\in \R_+\times \R_+,\  r+\lambda-\mu_1x_1-\mu_2x_2\leq  -\ee\},
  \end{align*}
  and
  \begin{equation*}
    A^0=\{(x_1,\, x_2)\in \R_+\times \R_+,\  r+\lambda-\mu_1x_1-\mu_2x_2=0\}.
  \end{equation*}
  According to (\ref{eq:3}), on $B_+^\ee$, the derivative of
  $\psi_1+\psi_2=\|(\psi_1,\, \psi_2)\|$ is greater than $\ee$, hence
  for a starting point in $A_+^\ee$, the trajectory has an $\L^1$
  increasing norm. Reasoning along the same lines on $B_-^\ee$, we
  see that for any $\eta>0$, for any starting point outside $A^0$, the
  trajectory of the differential system enters, in a finite time, one
  of the set $A_+^\eta$ or $A_-^\eta. $ Moreover, upon this time, the
  orbit stays in the compact $A_+^\eta\cup A_-^\eta$ forever.  It
  follows that (see for instance \cite{MR1422255})
  \begin{equation*}
    \label{eq:5}
    \lim_{t\to+\infty}\text{dist}\Bigl((\psi_1(t),\,\psi_2(t)),\ A^0\Bigr)=0.
  \end{equation*}
  This implies that any invariant set $M$ must be included in $A^0$.
  We then seek for a maximal invariant set. It is given by the
  intersection of the sets $Z_i=\{(x_1,\, x_2),\, f_i(x_1,\,
  x_2)=0\}$, $i=1,\, 2.$ We then remark that this system of equation
  is equivalent to the system
  \begin{math}
    f_1+f_2=0
  \end{math}
  and $f_2=0.$ It turns out that
  \begin{equation*}
    (f_1+f_2)(x_1,\, x_2)= r+\lambda -\mu_1x_1-\mu_2x_2=0.
  \end{equation*}
  The equation $f_2(x_1,x_2)=0$ yields to
  \begin{equation*}
    x_1=\frac{\mu_2x_2^2-\lambda x_2}{\lambda(1-p_I)-(\alpha p+\mu_2)x_2}=h(x_2).
  \end{equation*}
  The variations of $h$ shows that $h$ is a strictly decreasing
  diffeomorphism from $I=[\lambda(1-p_I)/(\alpha p+\mu_2),\,
  \lambda/\mu_2]$ onto $\R_+$. Hence its reciprocal function is a
  decreasing diffeomorphism from $\R_+$ onto $I$.

  Assume first that $r>0$. Then $(\lambda+r)/\mu_2>\lambda/\mu_2$ and
  there exists one and only one equilibrium point whose coordinates
  $({\xi}_1,{\xi}_2)$ are thus given by the solution of (\ref{eq:2})
  -- see Figure \ref{fig:fixed} for an illustration.


\begin{figure}[!ht]
  \begin{center}
   \leavevmode
   \includegraphics*[width=0.8\textwidth]{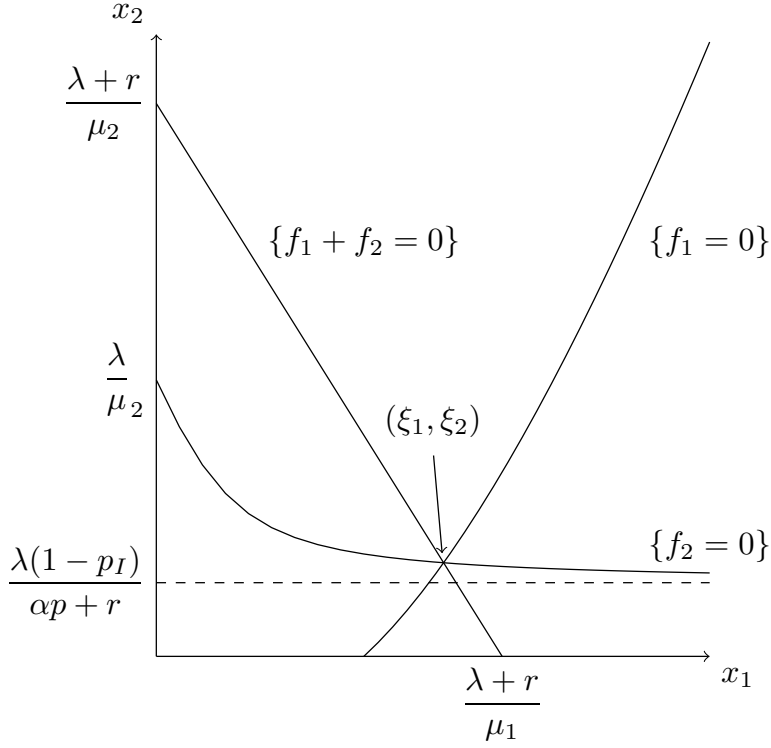}
   \caption{Determination of the fixed point.}
   \label{fig:fixed}
 \end{center}
\end{figure}

Consider the two distinct situations where $\mu_1=\mu_2$ and
$\mu_1\neq \mu_2$.  If $\mu_1\neq\mu_2$, then for any starting point
$(x_1,\, x_2)\neq (\xi_1,\, \xi_2)$ belonging to $A^0$,
$\psi_1^\prime(x_1,\, x_2)+\psi_2^\prime(x_1,\, x_2)=0$ but
$\mu_1\psi_1^\prime(x_1,\, x_2)+\mu_2\psi_2^\prime(x_1,\, x_2)\neq
0$. Hence for $t$ sufficiently close to $0$, $\psi(t)$ does not belong
to $A^0$ and then $(x_1,\, x_2)$ does not belong to $M$. Thus,
$M=\{(\xi_1,\, \xi_2)\}$ and according to the Poincar\'e-Bendixson
theorem (see \cite{MR1422255} for example),
\begin{equation}\label{eq:4}
  \lim_{t\to+\infty}(\psi_1(t),\psi_2(t))=(\xi_1,\, \xi_2).
\end{equation}
If $\mu_1=\mu_2$ then $\psi_1+\psi_2$ is solution of the differential
equation
\begin{equation*}
  v^\prime(t)=r+\lambda  -\mu_1 v(t),\ v(0)=x_1+x_2.
\end{equation*}
By direct integration, this yields to
\begin{equation*}
  (\psi_1+\psi_2)(t)=(x_1+x_2)e^{-\mu_1 t}+\frac{r+\lambda}{\mu_1}(1-e^{-\mu_1 t}).
\end{equation*}
This entails that $A^0$ is invariant. Since $A^0$ is compact, there
exists a minimum invariant set, say $M$. According to the
Poincar\'e-Bendixson theorem, $M$ is either a periodic orbit or a
critical point. Since $\psi_1+\psi_2$ is not periodic, $M$ is also
reduced to $(\xi_1,\, \xi_2)$ and we have (\ref{eq:4}).

For $r=0$, the point $(0,\, \lambda/\mu_2)$ is a fixed point. Due to
the concavity of $h^{-1}$, the sets $A^0$ and $\{f_2=0\}$ have at most
one point of intersection with positive abscissa. The existence of it
depends on the slope of $h^{-1}$ at the origin. By direct
computations, we find that
\begin{equation*}\label{eq:7}
  (h^{-1})^\prime(0)=-(\frac{\alpha p}{\mu_2}+p_I).
\end{equation*}
Hence there exists another equilibrium point if and only if $
(h^{-1})^\prime(0)>\mu_1/\mu_2$, i.e., ${\rho}=\alpha p +\mu_2
p_I-\mu_1>0$. We still denote by $(\xi_1,\, \xi_2)$ the unique
solution of (\ref{eq:2}) with a strictly positive first
coordinate. Note first that if $x_1=0$ then $\psi_1(t)=0$ for any $t$
thus the vertical axis is an invariant set. Moreover, for $x_1=0$, a
direct integration of (\ref{eq:1}) shows that
\begin{equation*}\label{eq:6}
  \lim_{t\to+\infty}(\psi_1(t),\psi_2(t))=(0,\lambda/ \mu_2).
\end{equation*}
We hereafter assume that $x_1\neq 0$.  If $\alpha p +\mu_2
p_I-\mu_1\le 0$, the same reasoning as above shows that
\begin{equation*}\label{eq:8}
  \lim_{t\to+\infty}(\psi_1(t),\psi_2(t))=(0,\, \lambda,/ \mu_2).
\end{equation*}
Assume now that $\alpha p +\mu_2 p_I-\mu_1> 0.$ At $(0,\,
\lambda/\mu_2),$ the linearization of (\ref{eq:1}) gives a matrix
whose
determinant
is given by
\begin{equation*}
  d=-{\rho}\mu_2.
\end{equation*}
Then, according to the hypothesis, $d<0$ 
thus $(0,\, \lambda/\mu_2)$ is a saddle point and cannot be an
attractor. Reasoning as above again yields to the conclusion that
every orbit converges to $ (\xi_1,\, \xi_2)$ for any $(x_1,\, x_2)$
such that $x_1\neq 0$.
\end{proof}
\section{Mean field approximation}
\label{sec:mean-field-appr}
We now consider a sequence $(X^N(t)=(X_1^N(t),\, X_2^N(t)),\, t\ge 0)$
of Markov processes with the same transitions as above but with
different rates given by (with self evident notations):
\begin{align*}
  q_1^N(n_1,\, n_2)&=r_N + \lambda_N\,p_I\dfrac{n_1}{n_1+n_2}\\
  q_2^N(n_1,\, n_2)&=\mu_1\, n_1\\
  q_3^N(n_1,\, n_2)&=\alpha\, p \,n_2\,\dfrac{n_1}{n_1+n_2}\\
  q_4^N(n_1,\, n_2)&=\lambda_N(1-p_I\dfrac{n_1}{n_1+n_2})\\
  q_5^N(n_1,\, n_2)&=\mu_2\, n_2.
\end{align*}
The main result of this Section is the following mean field
approximation of the system $X^N$.
\begin{theorem}
  \label{thm:mean-field-appr}
  Assume that
  \begin{equation*}
    \esp{\left\| \frac 1 N X^N(0)-x^0\right\|^2}\xrightarrow{N\to +\infty} 0,\
    \frac 1 N\, r_N \xrightarrow{N\to +\infty} r\ge 0, \
    \frac 1 N \,\lambda_N \xrightarrow{N\to +\infty} \lambda.
  \end{equation*}
  Let $\psi(x^0,.)=(\psi_1(x^0,.),\, \psi_2(x^0,.))$ be the solution of the differential
  system (\ref{eq:1}).
  Then, for any $T>0$,
  \begin{equation*}
    \esp{\sup_{t\le T}\left\| \frac 1 N X^N(t)-\psi(x^0,t)\right\|^2}\xrightarrow{N\to +\infty} 0.
  \end{equation*}
\end{theorem}
Before turning into the proof of Theorem \ref{thm:mean-field-appr},
let us give the martingale problem satisfied by the process $X^N$.
\begin{theorem}
  \label{thm:martingale problem}
  For any $N>0,$ the process $X^N$ is a vector-valued semi-martingale
  with decomposition:
  \begin{align*}
    X_1^N(t)=&X_1^N(0)+\int_0^t (q_1^N+q_3^N-q_2^N)(X^N(s))\d s+ M_1^N(t)\\
    X_2^N(t)&=X_2^N(0)+\int_0^t (q_4^N-q_3^N-q_5^N)(X^N(s))\d
    s+M_2^N(t),
  \end{align*}
  where $M^N=(M^N_1,\, M^N_2)$ is a local martingale vanishing at zero
  with square bracket given by:
  \begin{equation*}
    \<< M^N\>>_t=
    \begin{pmatrix}
      \displaystyle \int_0^t(q_1^N+q_3^N+q_2^N)(X^N(s))\d s  & \displaystyle -\int_0^t q_3^N(X^N(s))\d s\\
      &\\
      \displaystyle -\int_0^tq_3^N(X^N(s))\d s & \displaystyle
      \int_0^t(q_4^N+q_3^N+q_5^N)(X^N(s))\d s
    \end{pmatrix}.
  \end{equation*}
\end{theorem}
\begin{proof}
  Using the martingale problem associated with the Markov process
  $X^N$, we get that, for $t\geq 0$,
  \[
  X^N(t) = X^N(0) + \left(
    \begin{array}{c}
      \displaystyle \int_0^t (q_1^N+q_3^N-q_2^N)(X^N(s))\d s\\
      \displaystyle \int_0^t (q_4^N-q_3^N-q_5^N)(X^N(s))\d s
    \end{array}
  \right) +M^N(t),
  \]
  where $M^N=(M^N_1,\, M^N_2)$ is a 2-dimensional local martingale
  vanishing at zero.

  Let us now compute its square bracket. First of all, we consider
  $\langle M_1^N,M_2^N\rangle$. By integration by parts, we get that,
  for $t\geq 0$,
  \begin{align*}
    X_1^N(t)X_2^N(t) = X_1^N(0)X_2^N(0) &+
    \int_{(0,t]} X_1^N(s_-)\d X_2^N(s)\\
    &+ \int_{(0,t]} X_2^N(s_-)\d X_1^N(s) + [X_1^N,X_2^N]_t,
  \end{align*}
  where $[X_1^N,X_2^N]$ denotes the mutual variation of $X_1^N$ and
  $X_2^N$. Hence
  \begin{align*}
    X_1^N(t)X_2^N(t) =X_1^N(0)X_2^N(0)
    &+\int_0^t X_1^N(s)(q_4^N-q_3^N-q_5^N)(X^N(s))\d s \\
    &+\int_0^t X_2^N(s)(q_1^N+q_3^N-q_2^N)(X^N(s))\d s\\
    &+\,[X_1^N,X_2^N]_t\vphantom{\int_0^t}\\
    &+ \text{local martingale}.
  \end{align*}
  Now, writing the martingale problem associated with the process
  $X_1^NX_2^N$, we have
  \begin{align*}
    X_1^N(t)X_2^N(t) =X_1^N(0)X_2^N(0) &+\int_0^t X_1^N(s)
    (q_4^N-q_5^N)(X^N(s))\d s\\
    &+\int_0^t X_2^N(s)
    (q_1^N-q_2^N)(X^N(s))\d s\\
    &+\int_0^t
    (X_2^N(s)-X_1^N(s)-1)q_3^N(X^N(s))\d s\\
    &+\vphantom{\int_0^t} \text{local martingale}.
  \end{align*}
  We conclude that
  \[
  \la X_1^N,X_2^N\ra_t=-\int_0^t q_3^N(X_N(s))\d s.
  \]
  Similar arguments show that
  \[
  \la X_1^N\ra _t=\int_0^t(q_1^N+q_3^N+q_2^N)(X^N(s))\d s \text{ and
  }\la X_2^N\ra _t=\int_0^t(q_4^N+q_3^N+q_5^N)(X^N(s))\d s
  \]
  which ends the proof.
\end{proof}
\begin{proof}[Proof of Theorem \protect\ref{thm:mean-field-appr}]
According to Theorem   \ref{thm:asymptotique_syst_diff}, 
for any $x^0\in {\mathbf R}_+ \times {\mathbf R}_+\setminus \{(0,0)\}$
$\inf_{ s \in {\mathbf R}_+}\|\psi(x^0,s)\| >0.$
Then, the theorem \ref{thm:mean-field-appr} is a consequence of the following Lemma.
\end{proof}

\begin{lemma}\label{lem:control-ci}
There exists a constant $C$ depending only on $r,$ $\lambda,$ $P_I,$ $\mu_1,$
$\mu_2$ and $\alpha p$ such that for any $x^0 \in {\mathbf R}_+ \times {\mathbf R}_+ \setminus \{(0,0)\},$
 any $(X^M(0))_{M \in {\mathbf N}}$ sequence of random variables taking its values in $ {\mathbf R} \times {\mathbf R} \setminus \{(0,0)\},$  for any $N \in {\mathbf N}^*$, and  for any $T>0$
 \begin{multline*}
 {\mathbf E}\left[\sup_{t \leq T} \left\| \frac{1}{N}X^N(t) -
     \psi(x^0,t) \right\|\, \biggl | \, \sigma(X^M(0),\, M \in \N)\right]\\
 \leq  \left( \left\| \frac{1}{N}X^N(0) -x^0 \right\|^2 + \frac{1}{N} \left(T + T^2 \frac{1}{N} \|X^N(0)\|\right) \right)\\
 \times \exp \left(T\int_0^T ( 1 + \frac{1}{\|\psi(x^0,s)\|})^2 ds\right).
 \end{multline*}
 \end{lemma}
 \begin{proof}[Proof of Lemma \ref{lem:control-ci}]
   Let us fix $T>0$. Using Theorem~\ref{thm:martingale problem}, we
  have
  \begin{align*}
    \frac{1}{N}X_1^N(t)&=\frac{1}{N}X_1^N(0)+\int_0^{t} \frac{1}{N}(q_1^N+q_3^N-q_2^N)(X^N(s))\d s+ \frac{1}{N}M_1^N(t),\\
    \frac{1}{N}X_2^N(t)&=\frac{1}{N}X_2^N(0)+\int_0^{t} \frac{1}{N}
    (q_4^N-q_3^N-q_5^N)(X^N(s))\d s+\frac{1}{N}M_2^N(t).
  \end{align*}
  Moreover,
  \begin{align*}
    {\psi}_1(t)&= \int_0^t \left(q_1+q_3-q_2\right)({\psi}(s)) \d s,\\
    \psi_2(t)&=\int_0^t \left(q_4-q_3-q_5\right)({\psi}(s)) \d s.
  \end{align*}
  Note that for $x=(x^1,x^2)$ and $y= (y_1, y_2)$ in ${\mathbf R}_+ \times {\mathbf R}_+ \setminus \{(0,0)\},$
  then
  \begin{align*}
\left|\frac{x_1}{x_1+x_2}-\frac{y_1}{y_1+y_2}\right|
&\leq \left| \frac{x_1-y_1}{y_1+y_2} \right| + \left| \frac{x_1}{x_1+x_2} - \frac{x_1}{y_1 + y_2}\right| \\
&= \left| \frac{x_1-y_1}{y_1 +y_2} \right| + \left| \frac{x_1}{x_1 + x_2} \frac{ y_1-x_1 + y_2 -x_2}{y_1 +y_2} \right|\\
&\leq  2 \frac{\|x-y\|}{\|y\|}.
\end{align*}
We also have 
\begin{equation*}
\left| \frac{x_1x_2}{x_1+x_2} - \frac{y_1y_2}{y_1 + y_2} \right|\leq 2\| x- y\|.
\end{equation*}

  From now on, we use $\C$ for positive constants which depend only on
  $r$, ${\lambda}$, $p_I$, ${\mu}_1$, ${\mu}_2$ and ${\alpha}p$, and
  which may vary from line to line.  For $0\leq t\leq T$,

\begin{align}
  \nonumber
  &\left\| \frac{1}{N}X^N(t)-\psi(x^0,t)\right\|^2\\
  \label{eq:majo ZN}
 \leq  &\C \left( \left\|\frac{1}{N}X^N(0)-\psi(x^0,0)\right\|^2 +
    T^2\left|r-\frac{r_N}{N}\right|^2 +
    T^2\left|{\lambda}-\frac{{\lambda}_N}{N}\right|^2
  \right.\\
  \nonumber &\qquad+ \left.  T\int_0^t \left(1 + \frac{ 1}{\|\psi(x^0,s)\|}\right)^2\left\|
      \frac{1}{N}X^N(s)-\psi(x^0,s)\right\|^2\d s + \frac{1}{N^2}
    \left\|M^N(t)\right\|^2 \right).
\end{align}
Using Burkholder-Davis-Gundy inequality, we get that
\begin{align*} {\mathbb E}[\sup_{t\in [0,T
    ]}\left\|M^N(t)\right\|^2 | \sigma (X^M(0),~~M \in {\mathbf N})]
  &\leq \C {\mathbb E}[\left|\<< M^N\>>_{T}\right|  |\sigma (X^M(0),~~M \in {\mathbf N})].
\end{align*}
As a consequence of Lemma~\ref{lem:borne_sup}  we get that for $i\in
\left\{1,2,3,4,5\right\}$,
\begin{equation}
  \label{eq:10}
  \esp{\sup_{t\le T} q_i^N(X_s)}\leq \C ( \|X^N(0)\|+(r_N+\lambda_N)T),
\end{equation}
and
\begin{align*}
{\mathbb E}[\left|\<< M^N\>>_{T}\right|  | \sigma (X^M(0),~~~M \in {\mathbf N})] \leq C T( \esp{\|X^N(0)\|}+(r_N+\lambda_N)T).
\end{align*}
Hence, using Gronwall's lemma, (\ref{eq:majo ZN}) implies that
\begin{multline*}
 {\mathbf E}\left[\sup_{t \leq T} \left\| \frac{1}{N}X^N(t) - \psi(x^0,t) \right\| \, \biggl | \,\sigma(X^M(0),~~ M \in \N)\right]\\
 \leq  \left( \left\| \frac{1}{N}X^N(0) -x^0 \right\|^2 + \frac{1}{N} \left(T + T^2 \frac{1}{N} \|X^N(0)\|\right) \right)\\
 \times \exp \left(T\int_0^T ( 1 + \frac{1}{\|\psi(x^0,s)\|})^2 ds\right).
 \end{multline*}

\end{proof}

\section{Stationary regime}
\label{sec:stationary-regime}

We have proved so far that the process $N^{-1}X^N$
converges, as $N$ goes to infinity, to a deterministic $\R^2$-valued
function. This function converges, as $t$ goes to infinity, to a fixed
point $\psi^\infty$. On the other hand, for each $N$, the Markov process $X^N$
is ergodic thus has a limiting distribution as $t$ goes to
infinity. This raises the natural question to know whether this
limiting distribution converges to the Dirac mass at $\psi^\infty$
when $N$ goes to infinity.  Let us denote by $\P_{Y^N,\, \nu}$ the
distribution of the process $Y^N=N^{-1}X^N$ under initial distribution
$\nu.$ We denote by $\P_{\psi,\, \nu}$ the distribution of the process
whose initial state is chosen according to $\nu$ and whose
deterministic evolution is then given by the differential system
(\ref{eq:1}). According to Theorem \ref{thm:ergodic}, we know that
$X^N$ has a stationary probability whose value is irrespective of the
initial distribution of $X^N$. We denote by $Y^N(\infty)$ a random
variable whose distribution is the stationary measure of $Y^N$.  We
already know that
\begin{equation*}
  \begin{CD}
    \P_{Y^N(t),\delta_{x^0}} @>N\to \infty>>\P_{\psi(t),\, \delta_{x^0}}\\
    @V{t\to \infty}VV  @VV{t\to \infty}V\\
    \P_{Y^N(\infty)} @>?>{N\to \infty}> \delta_{\psi^\infty}
  \end{CD}
\end{equation*}
The question is then to prove that this is a commutative diagram,
i.e., that $Y^N(\infty)$ converges in distribution to the Dirac
measure at the equilibrium point of the system (\ref{eq:1}).  We
borrow the proof from \cite{MR812939} and \cite{MR1763158} but we need
to take into consideration the special role of the point $(0,0)$ which is a singular point for some
of the $q_j$.
\begin{defn}
  We say that a probability measure $\nu$ on $\R_+  \times \R_+ \setminus \{(0,0)\}$ belongs to $\Proba0$ when
  $\nu(\{0,\, 0\})=0$.
\end{defn}  
We will show that 1) for any sequence of initial distribution
$\nu^N$ converging weakly to $\nu$ with $\nu\in \Proba0$ then $\P_{Y^N, \, \nu^N}$ converges
weakly to $\P_{\psi,\, \nu}$, 2) that for any probability measure
$\nu\in \Proba0$, $\P_{\psi(t),\, \nu}$ converges weakly to
$\delta_{\psi^\infty}$, 3) that the sequence $(Y^N(\infty), \,
N\ge 1)$ is tight and 4) that any possible accumulation point of  $(Y^N(\infty), \,
N\ge 1)$ belongs to $\Proba0$.

The proof is then short and elegant: since $(Y^N(\infty),\, n\ge 1)$
is tight, it is sufficient to prove that there is a unique possible
limit to any convergent sub-sequence of $(Y^N(\infty))$. We still
denote by $Y^N(\infty)$ such a converging sub-sequence (as $N$ goes to
infinity). Its limit is denoted by $\nu,$ known to belong to $\Proba0$. According to Point 1. above,
$\P_{Y^N, \, \P_{Y^N(\infty)}}$ converges weakly to $\P_{\psi,\,
  \nu}$. Moreover by the properties of Markov processes, $\P_{Y^N, \,
  \P_{Y^N(\infty)}}$ is the distribution of a stationary process,
hence $\psi$ is also a stationary process when started from
$\nu$. This means that the distribution of $\psi(t)$ is $\nu$ for any
$t$. Then, by Point 2. above, $\nu=\delta_{\psi^\infty}.$ We have
thus proved that any convergent sub-sequence of $Y^N(\infty)$
converges to $\delta_{\psi^\infty},$ hence the result. We now turn to
the proof of the three necessary lemmas. 
\begin{theorem}
  For any sequence of initial distribution $\nu^N$ converging weakly to
  $\nu\in \Proba0$, then $\P_{Y^N, \, \nu^N}$ converges weakly to
  $\P_{\psi,\, \nu}$.
\end{theorem}
\begin{proof}
  We will proceed in two steps: First prove the tightness in $\D([0,\,
  T],\, \R^2)$ and then identify the limit. Actually, we will prove
  the slightly stronger result that $\P_{Y^N, \, \nu^N}$ is tight and
  that the limiting process is continuous.  According to
  \cite{MR1700749}, we need to show that for each positive $\epsilon$
  and $\eta$, there exists $\delta>0$ and $n_0$ such that for any
  $N\ge n_0$,
  \begin{equation*}
    \P\left( \sup_{\substack{|v-u|\le\delta\\v,u\le T}}\|Y^N(v)-Y^N(u)\|\ge
      \epsilon\right)\le \eta.
  \end{equation*}
  We denote by
  \begin{align*}
    A^N_1(t)&=\frac{1}{N}\int_0^t (q_1^N+q_3^N-q_2^N)(X^N(s))\d s\\
    A^N_2(t)&=\frac{1}{N}\int_0^t (q_4^N-q_3^N-q_5^N)(X^N(s))\d s.
  \end{align*}
  From Theorem \ref{thm:martingale problem}, we know that
  \begin{equation*}
    Y^N_i(v)=A^N_i(v)+\frac{1}{N}M^N_i(v),\ i=1,\,2.
  \end{equation*}
  Hence, for any positive $a$,
  \begin{multline}
    \label{eq:11}
    \P\left( \sup_{\substack{|v-u|\le\delta\\ v,u\le
          T}}\|Y^N(v)-Y^N(u)\|\ge
      \epsilon\right) \le \P(\|Y^N(0)\|\ge a)\\
    \begin{aligned}
      &+\P( \sup_{\substack{|v-u|\le\delta\\ v,u\le
          T}}\|A^N(v)-A^N(u)\|\ge
      \epsilon/2;\ \|Y^N(0)\|\le a)\\
      &+\P( \sup_{\substack{|v-u|\le\delta\\ v,u\le
          T}}\frac{1}{N}\|M^N(v)-M^N(u)\|\ge \epsilon/2;\ \|Y^N(0)\|\le
      a).
    \end{aligned}
  \end{multline}
  Eqn. (\ref{eq:10}) implies that
  \begin{equation*}
    \esp{\sup_{s\le T}\frac{1}{N^2}\|M^N(s)\|^2\,\biggl |\,\|Y^N(0)\|\le a }\le \frac{C(a+1)}{N}.
  \end{equation*}
  This means that $(N^{-1}M^N,\, N\ge 1)$ converges to $0$ in
  $L^2(\Omega; \D([0,T],\R^2),\, \P_{.|\|Y^N(0)\|\le a})$. Hence it
  converges in distribution in $\D([0,T],\R^2)$ and thus it is
  tight. This means that the last summand of (\ref{eq:11}) can be made
  as small as needed for large $N$. Furthermore,
  \begin{align*}
    \|A^N(v)-A^N(u)\|&\le \frac{2}{N}\int_u^v \sum_{i=1}^5 q_i(X^N(s))ds\\
    &\le 2|v-u|\left(\frac{r_N+\lambda_N}{N}+\frac CN \sup_{s\le T}
      \|X^N(s)\|\right).
  \end{align*}
  It follows from Lemma \ref{lem:borne_sup} that
  \begin{align*}
    \esp{\sup_{\substack{|v-u|\le\delta,\\ v,u\le T}}\|A^N(v)-A^N(u)\|\,\biggl |\,\|Y^N(0)\|\le
      a }&\le C\delta \,(\frac{r_N+\lambda_N}{N}T+a)\\
&\le C ((r+\lambda)T+a)\delta.
  \end{align*}
  This means that the second summand of (\ref{eq:11}) can also be made
  as small as wanted. The hypothesis on the initial condition exactly
  means that this also holds for the first summand of
  (\ref{eq:11}). Thus we have proved so far that $\P_{Y^N, \, \nu^N}$
  is tight and that its limit belongs to the space of continuous
  functions.

  We now prove that the only possible limit is $\P_{\psi,\,\nu}$. Assume that $\nu^N$ tends to $\nu$ and that $\P_{Y^N,\,
    \nu^N}$ tends to some $\P_{Z,\, \nu}$. We suppose that the initial conditions  $X^N(0)$ of the Markov processes are distributed as $\nu_N$
    and we introduce a random variable $x^0$ distributed as  $\nu.$ Recall that $Y^N=N^{-1} X^N.$ We fix  $M \in {\mathbf N}^*,$ $(\alpha^k=(\alpha_1^k, \alpha^k_2))_{0\leq k\leq M}\in {\mathbf R}^{2M+2}$ and $ 0=t_0 \leq t_1 \leq ... \leq t_M$. We introduce 
    \begin{align*}
    G^N&= {\mathbf E} \left( \exp i \left[\sum_{k=0}^M < \alpha_k, Y^N(t_{k})> \right] \right),\\
    \tilde{G}^N&={\mathbf E} \left( \exp i \left[\sum_{k=0}^M < \alpha_k, \psi(\frac{X^N(0)}{N},t_{k})> \right] \right),\\
     G&= {\mathbf E} \left( \exp i \left[\sum_{k=0}^M < \alpha_k, \psi(x^0,t_{k})> \right] \right),
    \end{align*}
    where $Y^N(-1)=0,$ and $Y^N=N^{-1} X^N,$ with initial condition $X^N(0)$ distributed as $\nu_N$
    and $X^0$ as $\nu.$\\
    
    Let $\varepsilon>0$. The sequence $(\nu_N)_{N \in {\mathbf N}}$ is tight, hence there exits  a compact set $K\subset {\R}_+ \times {\R}_+ \setminus \{(0,0)\}$ such that $\nu(K^c)+ \sup_N \nu_N( K^c) \leq \varepsilon$.
    We also introduce
    \begin{align*}
    G^N_K&= {\mathbf E}
     \left( \exp i
      \left[\sum_{k=0}^M < \alpha_k, 
      Y^N(t_{k})> \right]
       {\mathbf 1}_K(\frac{X^N(0)}{N})\right),\\
\tilde{G}^N_K&={\mathbf E} \left( \exp i \left[\sum_{k=0}^M < \alpha_k, \psi(\frac{X^N(0)}{N},t_{k})> \right]
    {\mathbf 1}_{K}(\frac{X^N(0)}{N} )\right),\\
     G_K&= {\mathbf E} \left( \exp i \left[\sum_{k=0}^M < \alpha_k, \psi(x^0,t_{k})> \right]{\mathbf 1}_K(x^0) \right).
    \end{align*}
    Then, \begin{align*}
    \limsup_N \left| G - G^N\right| \leq 2 \varepsilon + \limsup_N|\tilde{G}^N_K - G_K^N|+ \limsup_N|\tilde{G}^N_K - G_K^N|.
    \end{align*}

    From Theorem \ref{thm:asymptotique_syst_diff}, the map $(x,s) \mapsto \psi(x,s)$ is continuous on $({\R}_+ \times {\R}_+ \setminus \{(0,0)\}) \times [0,T]$ and 
    $\inf_{ (x,s) \in K \times [0,T]} \| \psi(x,s) \|>0.$
    Since $\frac{ X^N}{N}(0)$ takes is value in the  compact set $K$,
    then from Lemma \ref{lem:control-ci}, $\limsup_N|\tilde{G}^N_K - G_K^N|=0.$
    Since the sequence of measures $(\nu_N)$ converges weakly to $\nu,$ then $\limsup_N|\tilde{G}^N_K - G_K^N|=0$. Hence,
\begin{align*}
\limsup_N \left| G - G^N\right| \leq 2 \varepsilon 
\end{align*}
 for all $\varepsilon >0.$

 That means for any $t_0,\cdots, t_M$,
  \begin{equation*}
    {\mathbf P}_{(Y^N(t_0 ),\cdots,Y^N(t_M)),\nu_N }\xrightarrow{N\to 0} \P_{(\psi(x^0,t_0),\cdots,\psi(x^0,t_N)),\nu},
  \end{equation*}
  Hence all the accumulation points are the same and the convergence
  of $\P_{Y^N,\nu^N}$ follows.
\end{proof}
\begin{theorem}
  For any probability measure $\nu\in \Proba0$, $\P_{\psi(t),\, \nu}$
  converges weakly to $\delta_{\psi^\infty}$ as $t\to \infty.$
\end{theorem}
\begin{proof}
  For any $f$ continuous bounded on $\R^2$, we have
  \begin{equation*}
    \int f d\P_{\psi(t),\, \nu}=\int_{\R^2}\esp{f(\psi(t))\,|\, \psi(0)=x}\d\nu(x).
  \end{equation*}
  Theorem \ref{thm:asymptotique_syst_diff} says that for any $x\in
  \R_+\times\R_+\setminus\{(0,0)\}$,
  \begin{equation*}
    \esp{f(\psi(t))\,|\, \psi(0)=x} \xrightarrow{t\to \infty} f(\psi^\infty).
  \end{equation*}
  The result follows by dominated convergence.
\end{proof}
\begin{theorem}
  \label{thm:tightness}
  The sequence $(Y^N(\infty), \, N\ge 1)$ is tight and any
  accumulation point belongs to $\Proba0$.
\end{theorem}
We need a preliminary lemma which relies on the observation that when
$\mu_1=\mu_2$, the process $X_1+X_2$ has the dynamics of the process
counting the number of customers in an M/M/$\infty$ queue. Recall that
$\mu_-=\min(\mu_1,\, \mu_2)$ and set  $\zeta=(r+\lambda p_I)/\mu_-$. For any
$c\in \R_+,$ any $x\in \N$, define the function
\begin{equation*}
  h_c(t,\, x)=(1+ce^{\mu_- t})^xe^{-\zeta c\exp(\mu_- t)}.
\end{equation*}
Note that $h_c$ is increasing with respect to $x$. Moreover, according
to \cite[Chapter 6]{MR1996883},
\begin{equation}\label{eq:9}
  \frac{\partial h_c}{\partial t}(t,\, x)+R(h_c(t,\, .))(x)=0,
\end{equation}
where, for any $w\, :\, \N\to \R,$
\begin{equation*}
  Rw(x)=(w(x+1)-w(x))(r+\lambda p_I)+(w(x-1)-w(x))\mu_-x.
\end{equation*}
\begin{lemma}
  \label{lem:supermartingale}
  For any non negative real $c$, the process $H_c=(h_c(t,X_1(t)+X_2(t)), \, t\ge
  0)$ is a positive supermartingale.
\end{lemma}
\begin{proof}
  According to Dynkin
  formula (see \cite[Proposition C.5]{MR1996883}), for any $0\le s <
  t$, we have
  \begin{multline*}
    0={\mathbf E}\left[\vphantom{\int_s^t \frac{\partial h_c}{\partial r}}h_c(t,\|X(t)\|)-h_c(s,\|X(s)\|) -\int_s^t \frac{\partial h_c}{\partial t}(r,\, \|X(r)\|)\right.\\
    \begin{aligned}
      &-(r+\lambda p_I)\int_s^t \Bigl(h_c(r,\|X(r)\|+1)-h_c(r,\|X(r)\|)\Bigr)\d r\\
      &\left.-\int_s^t \Bigl(h_c(r,\|X(r)\|-1)-h_c(r,\|X(r)\|)\Bigr)(\mu_1 X_1(r)+\mu_2X_2(r))\d r\, \biggl|\, {\mathcal F}_s\right]  \\
      \ge &{\mathbf E}\left[\vphantom{\int_s^t \frac{\partial h_c}{\partial r}}h_c(t,\|X(t)\|)-h_c(s,\|X(s)\|)-\int_s^t \frac{\partial h_c}{\partial t}(r,\, \|X(r)\|)\right.\\
      &-(r+\lambda p_I)\int_s^t \Bigl(h_c(r,\|X(r)\|+1)-h_c(r,\|X(r)\|)\Bigr)\d r\\
      &\left.-\int_s^t \Bigl(h_c(r,\|X(r)\|-1)-h_c(r,\|X(r)\|)\Bigr)\mu_-(
        X_1(r)+X_2(r))\d r\,\biggl |\, {\mathcal F}_s\right],
    \end{aligned}
  \end{multline*}
  where the inequality follows from the monotony of $h_c$ and the
  definition of $\mu_-$. Hence we get that
  \begin{multline*}
    0\ge {\mathbf E}\left[\vphantom{\int_s^t \frac{\partial h_c}{\partial r}}h_c(t,\|X(t)\|)-h_c(s,\|X(s)\|)\right.\\
    \left.-\int_s^t \frac{\partial h_c}{\partial t}(r,\,
      \|X(r)\|)+R(h_c(r,.))(\|X(r)\|)\d r\,\biggl|\, {\mathcal
        F}_s\right].
  \end{multline*}
  In view of Eqn. (\ref{eq:9}), we get
  \begin{equation*}
    0\ge \esp{h_c(t,\|X(t)\|)-h_c(s,\|X(s)\|)\,|\, {\mathcal F}_s},
  \end{equation*}
  i.e., $H_c$ is a supermartingale.

Now, let  $(Y^N(\infty),\, N \geq 1)$ be a subsequence which converge to $\nu.$
  Since $X^N(\infty)$ is a random variable distributed according to
  the stationary law of the process ${X^N}$,
\begin{align*}
\esp{ Qe^{- \| .\|}(X^N(\infty))}=0.
\end{align*}
By a direct calculation, we have
\begin{align*}
Qe^{- \|.\|}(x)= e^{-\|x\|}[(\lambda +r)(e^{-1 } -1) + (\mu_1 x_1 + \mu_2 x_2)(e -1)],
\end{align*}
then
\begin{multline*}
(\lambda_N +r_N)(1-e^{-1 })\esp{ e^{-N \|Y^N(\infty)\|}}\\
= N(e -1) \esp{ e^{-N \|Y^N(\infty)\|}(\mu_1 Y_1^N(\infty) + \mu_2 X_2^N(\infty))}.
\end{multline*}
Hence,
\begin{math}
(1- e^{-1})(r+ \lambda)\nu(\{(0,0)\})=0,
\end{math}
i.e., $\nu$ belongs to $\Proba0$.
\end{proof}
\begin{proof}[Proof of Theorem \protect~\ref{thm:tightness}]
  Let $K$ be real, for any positive real $\theta$, we have
  \begin{equation*}
    \P(\|Y^N(t)\|>K)=\P(\|X^N(t)\|>NK)\le e^{-\theta NK}\esp{\exp(\theta \|X^N(t)\|)}.
  \end{equation*}
  Lemma \ref{lem:supermartingale} entails that
  \begin{equation*}
    \esp{\exp(\theta \|X^N(t)\|)}\le (1+(e^\theta-1)e^{-\mu_-t})^{N\|X^N(0)\|}\exp\Bigl(N\zeta
    (e^\theta -1)(1-e^{-\mu_-t})\Bigr).
  \end{equation*}
  Hence,
  \begin{multline*}
    \P(\|Y^N(\infty)\|>K)=\lim_{t\to \infty}\P(\|Y^N(t)\|>K)\\
    \begin{aligned}
      &\le \inf_{\theta>0}\lim_{t\to
        \infty}\left(1+(e^\theta-1)e^{-\mu_-t}\right)^{N\|X^N(0)\|}
      \exp\Bigl(-\theta NK+ N\zeta
      (e^\theta -1)(1-e^{-\mu_-t})\Bigr)\\
      &=\inf_{\theta>0}  \exp\left(N(-\theta K+ \zeta(e^\theta -1)\right)\\
      &\le \exp(-\frac{1}{2}NK\ln K),
    \end{aligned}
  \end{multline*}
for $K$ large enough.
  The tightness follows.
\end{proof}
\section{Central Limit Theorem}
\label{sec:centr-limit-theor}
It turns out that we can also evaluate the order of the approximation
when we replace $X^N$ by $\psi$. This is given by CLT like theorem.
\begin{theorem}
  Assume that the hypothesis of Theorem \ref{thm:mean-field-appr}
  holds. Then, for any $T>0$, the process
  \begin{equation*}
    W^N=\sqrt{N}(Y^N-\psi)
  \end{equation*}
  tends in distribution in $\D([0,\, T], \, \R^2)$ to a centered
  Gaussian process with covariance matrix $\Gamma(t)$ given by:
  \begin{equation*}
    \Gamma(t)=\begin{pmatrix}
      \Gamma_1(t)  &  - \alpha p\displaystyle\int_0^t\dfrac{\psi_1(s)\psi_2(s)}{\psi_1(s)+\psi_2(s)}\d s\\
     - \alpha p\displaystyle\int_0^t \dfrac{\psi_1(s)\psi_2(s)}{\psi_1(s)+\psi_2(s)}\d s & \Gamma_2(t)
    \end{pmatrix},
  \end{equation*}
  where
  \begin{align*}
    \Gamma_1(t)&= rt + \displaystyle\int_0^t \lambda\,p_I\dfrac{\psi_1(s)}{\psi_1(s)+\psi_2(s)}+\mu_1\, \psi_1(s)+\alpha\, p     \,\dfrac{\psi_1(s)\psi_2(s)}{\psi_1(s)+\psi_2(s)}\d s\\
    \Gamma_2(t)&= \displaystyle\int_0^t
    \lambda(1-p_I\dfrac{\psi_1(s)}{\psi_1(s)+\psi_2(s)})+\mu_2\,
    \psi_2(s)+\alpha\, p
    \,\dfrac{\psi_1(s)\psi_2(s)}{\psi_1(s)+\psi_2(s)}\d s.
  \end{align*}
\end{theorem}
\begin{proof}
  According to \cite[p. 339]{ethier86}, it suffices to prove that
  \begin{equation*}
    \esp{\sup_{t\le T}|W^N(t)-W^N(t_-)|}\xrightarrow{N\to +\infty} 0
  \end{equation*}
  and that
  \begin{equation*}
    \<< W^N\>>_t \xrightarrow{N\to +\infty} \Gamma(t).
  \end{equation*}
  Since the jumps of $Y^N$ are bounded by $1/N$, those of $W^N$ are
  bounded by $N^{-1/2}$, hence the first point is proved. As to the
  second point, remark that
  \begin{equation*}
    \<< W^N\>>_t = N^{-1} \<< M^N\>>_t
  \end{equation*}
  and then use Theorem \ref{thm:mean-field-appr}.
\end{proof}

\section{Numerical investigation}
\label{sec:numer-invest}
Another approach to evaluate the order of approximation can be made by
computer simulation. We simulated the Markov process for $N=100$ and
compute the estimate of the prevalence by a simple Monte-Carlo
method on $10,000$ trajectories. For  the parameters we chose, $\alpha =1,\, \mu_1=0.1,\, \mu_2=0.2,\, r=1,\, \lambda=5$ and
$p_I=0.8,$ the results are strikingly good as
shown in Figure \ref{fig:prevalence}. Note
that the choice of parameters is here very delicate since biological
parameters are not very well known (i.e., $\mu_1$, $\mu_2$, $p$,
$\ldots$) and population dependant quantities are even more obscure to
determine. We here chose parameters which seems reasonable and fit the
observed prevalence. 

\begin{figure}[!ht]
  \centering
  \includegraphics*[width=0.8\textwidth]{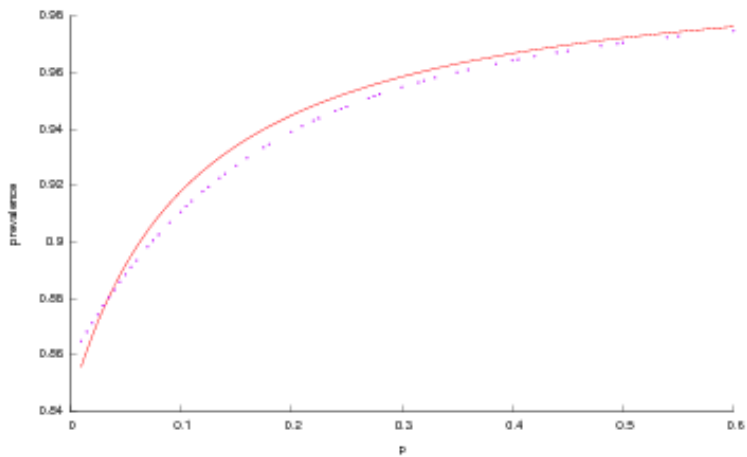}
  \caption{Prevalence with respect to $p$. The solid line represents
    the value as computed by Equations \ref{eq:2}. The dots represents
    the simulated values. The 95$\%$ confidence interval are so small, they can't be displayed.}
  \label{fig:prevalence}
\end{figure}

In such models, another quantity of interest is the relative
importance of each parameters: what does affect  most the
prevalence~? On the deterministic system, this question is easily
solved by computing the derivative of the prevalence with respect to
each of the parameters. We now explain how to compute the sensitivity
of the prevalence on the stochastic model. Say we have a function $F$
bounded which depends on the sample-paths of $X$, we aim to compute:
\begin{equation*}
  \frac{d}{dp}{\mathbf E}_p[F],
\end{equation*}
where we put a $p$ under the expectation symbol to emphasize the
dependence of the underlying probability with respect to $p$. Other
``greeks'', as these quantities are called in mathematical finance,
can be derived analogously. We
assume that we observe the Markov process on a time window of size
$T$, i.e., any functional is implicitly assumed to belong to
$\F_T=\sigma\{X(s),\, 0\le s\le T\}.$
\begin{theorem}
  For any bounded $F$, $F\in \F_T$, we have:
  \begin{align*}
 \frac{d}{dp}{\mathbf E}_p[F]&=\frac 1p{\mathbf E}_p\left[ F \left( \sum_{s\le
       T}\car_{\{(1,\, -1)\}}(\Delta X(s))-\int_0^T q_3(X(s_-))\d
     s\right)\right]\\
&=\frac 1p\text{cov}_p\left( F,\    \sum_{s\le
       T}\car_{\{(1,\, -1)\}}(\Delta X(s))\right),
  \end{align*}
where $\Delta X(s)=X(s)-X(s_-)$.
 \end{theorem}
 \begin{proof}
   The proof relies on the Girsanov theorem which is more easily
   expressed in the framework of multivariate point measures.
Since there are only five kind of jumps, we can represent the dynamics
of $X$ as a point measures on $\R^+\times \{1,\cdots,\, 5\}$:
\begin{equation*}
  \mu([0,\, t]\times \{i\})=\sum_{s\le t} \car_{\{\Delta X(s)=l_i\}},
\end{equation*}
where
\begin{equation*}
  l_1=(1,\, 0),\ l_2=(-1,\, 0),\, l_3=(1,\, -1),\, l_4=(0,\,1)
  \text{ and } l_5=(0,\, -1). 
\end{equation*}
In the reverse direction, 
\begin{equation*}
  X(t)=X(0)+\sum_{i=1}^5 \mu([0,\, t]\times \{i\}) \, l_i.
\end{equation*}
It is immediate from the preceding results that $\nu^p$,  the $\P_p$-predictable
compensator of $\mu$ is given by
\begin{equation*}
  \d\nu^p(t,\, i)=q_i(X(t_-))\d t.
\end{equation*}
To compute $ {d}/{dp}{\mathbf E}_p[F]$ means to compute 
\begin{equation*}
  \lim_{\ee \to 0} \frac 1\epsilon \left({\mathbf E}_{p+\ee}[F]-{\mathbf E}_p[F]\right).
\end{equation*}
Under $\P_{p+\ee}$, 
\begin{equation*}
   \d\nu^{p+\ee}(t,\, i)=\d\nu^p(t,\, i) \text{ for } i\neq 3 \text{ and }
   \d\nu^{p+\ee}(t,\, 3)=(1+\frac \ee p)  \d\nu^p(t,\, 3).
\end{equation*}
Let 
\begin{equation*}
  U(t,\, i)=
  \begin{cases}
    0 & \text{ if } i\neq 3,\\
\dfrac \ee p & \text{ if } i=3.
  \end{cases}
\end{equation*}
According to the Girsanov theorem (see
\cite{decreusefond95_3,jacod79}), this means that 
\begin{align*}
  {\mathbf E}_{p+\ee}[F]&={\mathbf E}_{p}\left[F\, {\mathcal E}(\int_0^t
    U(s,\, i)(\d\mu(s,\, i)-\d\nu^p(s, \, i))) \right]\\
&={\mathbf E}_{p}\left[F\, {\mathcal E}\left(\frac\ee p (\mu([0,\, T]\times
  \{3\})-\nu^p([0,\, T]\times\{3\}))\right),
    \right]
\end{align*}
where ${\mathcal E}$ denotes the Dol\'eans-Dade exponential. It is known
(see \cite{decreusefond95_3}) that a Doléans-Dade exponential follows
 the same rule of derivation as a usual exponential, hence
the result.
\end{proof}
With the parameters above, the simulated greek coincides pretty well
with the sensitivity computed by differentiating the expression of the
stationary prevalence in the deterministic system, see Figure \ref{fig:sensitivity}. However, as usual
with this method, the confidence interval are rather large.

\begin{figure}[!ht]
  \centering
  \includegraphics*[width=0.8\textwidth]{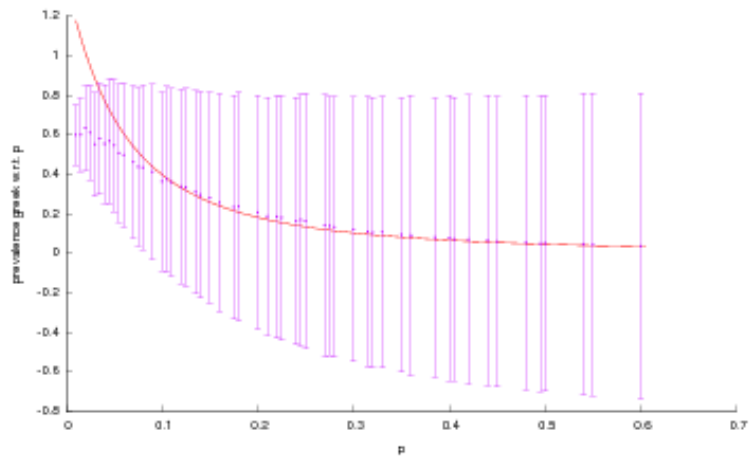}
  \caption{Prevalence greek with respect to $p$. Same conventions as
    above. The error bars represent the $95\%$ confidence interval.}
  \label{fig:sensitivity}
\end{figure}


\providecommand{\bysame}{\leavevmode\hbox to3em{\hrulefill}\thinspace}
\providecommand{\MR}{\relax\ifhmode\unskip\space\fi MR }
\providecommand{\MRhref}[2]{%
  \href{http://www.ams.org/mathscinet-getitem?mr=#1}{#2}
}
\providecommand{\href}[2]{#2}


\end{document}